\documentclass[10pt]{amsart}

\usepackage{amssymb, latexsym, amsfonts, pigpen, enumitem, mathrsfs, amsthm, bbm, stackrel}
\usepackage[matrix,arrow,curve]{xy}

\newtheorem{theorem}{Theorem}[section]
\newtheorem{lemma}[theorem]{Lemma}
\newtheorem{corollary}[theorem]{Corollary}

\theoremstyle{definition}
\newtheorem{definition}[theorem]{Definition}
\newtheorem{example}[theorem]{Example}

\theoremstyle{remark}
\newtheorem{remark}[theorem]{Remark}

\numberwithin{equation}{section}

\newcommand{\mc}[1]{\mathcal{#1}}

\newcommand{\Sp}{{\mathrm{Sp}}}
\newcommand{\SL}{{\mathrm{SL}}}
\newcommand{\SLc}{{\mathrm{SL^c}}}

\newcommand{\Hom}{{\mathrm{Hom}}}

\newcommand{\Z}{{\mathbb{Z}}}

\newcommand{\GW}{{\mathrm{GW}}}

\newcommand{\A}{\mathbb{A}}
\newcommand{\SSp}{\mathbb{S}}

\newcommand{\Th}{{Th}}
\newcommand{\thc}{{th}}
\newcommand{\id}{\operatorname{id}}
\newcommand{\Cone}{\operatorname{Cone}}

\newcommand{\triv}{\mathbf{1}}

\newcommand{\Gm}{{\mathbb{G}_m}}
\newcommand{\GL}{\mathrm{GL}}
\newcommand{\SH}{\mathcal{SH}}

\newcommand{\Ht}{\mathrm{H}}
\newcommand{\T}{\mathrm{T}}

\newcommand{\Spec}{\operatorname{Spec}}

\newcommand{\Sm}{\mathcal{S}m}

\newcommand{\cT}{\mathcal{T}}

\newcommand{\et}{\mathrm{\acute{e}t}}

\newcommand{\etale}{\'etale }

\newcommand{\DMW}{\mathrm{\widetilde{DM}}}
\newcommand{\DMWk}{\mathrm{\widetilde{DM}}(k)}
\newcommand{\DM}{\mathrm{DM}}
\newcommand{\DMk}{\mathrm{DM}(k)}

\newcommand{\DAS}{\mathrm{D}_{\A^1}(S)}

\newcommand{\ZS}{\mathrm{H}_{\mathbb{A}^1}\mathbb{Z}}
\newcommand{\ZStop}{\mathrm{H}_{\mathrm{top}}\mathbb{Z}}
\newcommand{\constf}{\mathcal{C}onst}
\newcommand{\etatop}{\eta_{\mathrm{top}}}

\newcommand{\SmS}{\mathcal{S}m_S}
\newcommand{\SHS}{\mathcal{SH}(S)}
\newcommand{\SHk}{\mathcal{SH}(k)}
\newcommand{\Smk}{\mathcal{S}m_k}

\newcommand{\motive}{\mathcal{M}}
\newcommand{\emclane}{\mathrm{EM}}
\newcommand{\homol}{\mathrm{C}}
\newcommand{\HZtop}{\mathrm{H_{top}}\mathbb{Z}}
\newcommand{\HM}{\mathrm{H_M}}
\newcommand{\HMW}{\mathrm{H_{MW}}}
\newcommand{\HAZ}{\mathrm{H}_{\mathbb{A}^1}\mathbb{Z}}

\newcommand{\mcE}{\mathcal{E}}
\newcommand{\Spt}{\mathrm{Spt}}

\newcommand{\pour}{\ar@{}[ur]|(0.2){\text{\pigpenfont G}}}
\newcommand{\podr}{\ar@{}[dr]|(0.2){\text{\pigpenfont A}}}

\newcommand{\leftrarrows}{\mathrel{\raise.75ex\hbox{\oalign{%
				$\scriptstyle\leftarrow$\cr
				\vrule width0pt height.5ex$\hfil\scriptstyle\relbar$\cr}}}}
\newcommand{\lrightarrows}{\mathrel{\raise.75ex\hbox{\oalign{%
				$\scriptstyle\relbar$\hfil\cr
				$\scriptstyle\vrule width0pt height.5ex\smash\rightarrow$\cr}}}}
\newcommand{\Rrelbar}{\mathrel{\raise.75ex\hbox{\oalign{%
				$\scriptstyle\relbar$\cr
				\vrule width0pt height.5ex$\scriptstyle\relbar$}}}}

\makeatletter
\def\leftrightarrowsfill@{\arrowfill@\leftrarrows\Rrelbar\lrightarrows}
\newcommand{\xleftrightarrows}[2][]{\ext@arrow 3399\leftrightarrowsfill@{#1}{#2}}
\makeatother

\begin{document}

\title{Thom isomorphisms in triangulated motivic categories}


\author{Alexey Ananyevskiy}
\address{St. Petersburg Department, Steklov Math. Institute, Fontanka 27, St. Petersburg 191023 Russia}
\email{alseang@gmail.com}


\date{}

\thanks{The research is supported by Young Russian Mathematics award, by RFBR grants~18-31-20044 and~19-01-00513 and by a grant from the Government of the Russian Federation, agreement~075-15-2019-1620}

\begin{abstract}
	We show that a triangulated motivic category admits categorical Thom isomorphisms for vector bundles with an additional structure if and only if the generalized motivic cohomology theory represented by the tensor unit object admits Thom classes. We also show that the stable $\A^1$-derived category does not admit Thom isomorphisms for oriented vector bundles and, more generally, for symplectic bundles. In order to do so we compute the first homology sheaves of the motivic sphere spectrum and show that the class in the coefficient ring of $\A^1$-homology corresponding to the second motivic Hopf map $\nu$ is nonzero which provides an obstruction to the existence of a reasonable theory of Thom classes in $\A^1$-cohomology.
\end{abstract}

\maketitle

\section{Introduction}

With the development of the theory of motivic complexes nowadays there is a number of triangulated and stable $\infty$-categories of motivic nature which allow one to study various cohomological properties of algebraic varieties. Most of these categories arise via the following two constructions.
\begin{itemize}
	\item 
	Let $\mathrm{Cor}_S^{\mc{T}}$ be a category of correspondences. Then the associated triangulated motivic category is the category of complexes of sheaves of abelian groups (or the category of $S^1$-spectra of simplicial sets) over $\mathrm{Cor}_S^{\mc{T}}$ with inverted $\A^1$-equivalences and inverted $\Gm$-suspension functor.
	\item 
	Let $A\in\SH(S)$ be a highly structured commutative ring spectrum. Then the associated triangulated motivic category is the homotopy category of modules over $A$.
\end{itemize}
These constructions are closely related, see \cite[Theorem~1]{RO08}, \cite{EK17} and \cite[Theorem~5.3]{G19}. Both the constructions exhibit a tensor triangulated category $\mc{T}$ together with an adjunction
\[
\gamma^*_\cT\colon \SHS \leftrightarrows \cT\colon \gamma_*^\cT
\]
with $\gamma_\cT^*$ being strict monoidal. Here $\SHS$ is the stable motivic homotopy category introduced by Morel and Voevodsky \cite{MV99,V98} which is a universal triangulated motivic category. We refer to such tensor triangulated categories with an adjunction as triangulated motivic categories over $S$.

One may study triangulated motivic categories $\mc{T}$ over $S$ combining geometrical constructions carried out in $\SHS$ with the properties of the bigraded cohomology theory represented by 
\[
\Ht_{\mc{T}} \Z= \gamma_*^\cT\gamma^*_\cT \SSp \in \SHS
\]
where $\SSp\in\SHS$ is the sphere spectrum. In the current paper we address the question of the existence of Thom isomorphisms and Thom classes for vector bundles possibly carrying an additional structure of a symplectic bundle (a symplectic form), an $\SL$-bundle (a trivialization of the determinant) or an $\SLc$-bundle (a choice of the square root of the determinant). In particular we show the following.
\begin{theorem}[Theorem~\ref{thm:thom_iso_equiv}]
	For a triangulated motivic category $\mc{T}$ over $S$ the following are equivalent (see Definitions~\ref{def:motive_Thom}, \ref{def:coh_Thom} and~\ref{def:thom_cat} for the notation):
	\begin{enumerate}
		\item $\cT$ admits normalized Thom isomorphisms
		\[
		\motive_{\cT}(\Th(\mc{E}))\cong \Sigma_{\T}^n \motive_{\cT}(X)
		\]		
		for vector bundles (resp. symplectic bundles, vector $SL$-bundles, vector $SL^c$-bundles),
		\item the cohomology theory represented by the spectrum $\Ht_\cT \Z\in \SHS$ admits normalized Thom classes
		\[
		\thc(\mcE)\in (\Ht_\cT \Z)^{2n,n}(\Th(\mcE))
		\]
		for vector bundles (resp. symplectic bundles, vector $SL$-bundles, vector $SL^c$-bundles).
	\end{enumerate}
\end{theorem}
\noindent As a consequence we obtain 
\begin{corollary}[Corollary~\ref{cor:cat_orient_sheaf}]
	Let $\cT$ be a triangulated motivic category over $S$ and suppose that the presheaf $(\Ht_{\cT} \Z)^{0,0}(-)$ is a Zariski sheaf on $\SmS$. Then $\cT$ admits normalized Thom isomorphisms for vector $SL^c$-bundles (i.e. for vector bundles admitting a square root of the determinant).
\end{corollary}
\noindent In particular, this recovers the existence of Thom isomorphisms in the triangulated category of MW-motives $\DMWk$ \cite[Theorem~6.2]{Y19}. Note also that homotopy purity isomorphisms yield that the existence in $\mc{T}$ of Thom isomorphisms for vector bundles with an additional structure is equivalent to the existence of Gysin triangles
\[
\motive_\cT (X-Z) \to \motive_\cT (X) \to \Sigma^n_{\cT}\motive_\cT (Z) \to \motive_\cT (X-Z)[1]
\]
for closed embeddings $Z\to X$ of smooth schemes with the normal bundle $N_{Z/X}$ admitting the corresponding additional structure (see Remark~\ref{rem:Gysin}).

We also show that the stable $\A^1$-derived category which is the universal linear triangulated motivic category does not admit Thom isomorphisms for vector bundles with an additional structure.
\begin{theorem}[Theorem~\ref{thm:DAS_Thom}]
	$\DAS$ does not admit normalized Thom isomorphisms for vector bundles and for vector bundles with an additional symplectic, $SL$ or $SL^c$-structure.
\end{theorem}
\noindent Note that if one imposes \'etale descent and considers the category of \'etale motivic sheaves $\mathrm{D}_{\A^1,\et}(S)$ then one has normalized Thom isomorphisms for vector bundles by \cite[Remarque~11.3]{Ayo14} assuming that $S$ satisfies the technical condition of \cite[Hypoth\`ese~7.3]{Ayo14}.

In order to prove the above theorem we compute the first homology sheaves of the motivic sphere spectrum combining results of \cite{RSO19} and the following theorem.
\begin{theorem}[Theorem~\ref{thm:a1homology}]
	Let $S= \Spec k$ be the spectrum of a field. Then the following sequence is exact:
	\[
	\underline{\pi}_0(\SSp)_*\xrightarrow{\cup\eta^{\A^1}_{\mathrm{top}}} \underline{\pi}_1(\SSp)_* \xrightarrow{H} \underline{\pi}_1(\ZS)_*\to 0.
	\]
	Here $\ZS=\Ht_{\DAS}\Z$ and $H$ is induced by the unit morphism $\SSp\to \ZS$.
\end{theorem}
\noindent 
It follows that the element in the coefficient ring of $\ZS$ corresponding to the second Hopf map $\nu\in \underline{\pi}_{1}(\SSp)_{-2}(S)$ is nontrivial which turns out to be an obstruction to the existence of Thom isomorphisms (Lemma~\ref{lem:nusymp}).

The paper is organized as follows. In Section~\ref{section:triangulated} we introduce the formalism related to triangulated motivic categories. In the next section we recall the notions of vector bundles with additional structure, Thom classes, and Thom isomorphisms, and then we show that the existence of Thom classes is equivalent to the existence of categorical Thom  isomorphisms. In Section~\ref{section:homology} we show how to compute the first homology sheaves of the motivic sphere spectrum out of the zeroth and first homotopy sheaves and show that the class $H(\nu)$ in $\A^1$-homology corresponding to the second Hopf map $\nu$ is nontrivial. In the last section we show that nontriviality of $H(\nu)$ gives an obstruction to the existence of Thom isomorphisms in the stable $\A^1$-derived category $\DAS$.

{\it Acknowledgement.} I would like to thank Vladimir Sosnilo for enlightening discussions which eventually led to Theorem~\ref{thm:a1homology} and Oliver R\"ondigs for answering many questions related to the Hopf map $\nu$ and the first homotopy module of the motivic sphere spectrum.

Throughout the paper we employ the following assumptions and notations.

\begin{tabular}{l|l}
	$S$ & a quasi-compact separated scheme \\
	$\SmS$ & the category of smooth schemes over $S$\\
	$\SHS$ & the motivic stable homotopy category over $S$ \cite{MV99,V98}\\
	$\DAS$ & the stable $\A^1$-derived category over $S$ \cite[Chapter~5]{CD19}\\	
	$\SSp$ & the motivic sphere spectrum in $\SHS$\\	
	$\Th(E)$ & the Thom space of a vector bundle $E$\\
	$\Gm$ & $\A^1-0$ considered as a scheme over $S$\\
	$\triv_X^{\oplus n}$ & the trivialized rank $n$ vector bundle over $X\in\SmS$
\end{tabular}

\section{Preliminaries on triangulated motivic categories} \label{section:triangulated}

\begin{definition}
	A \textit{triangulated motivic category over $S$} is a tensor triangulated category $\cT$ equipped with an adjunction of triangulated categories 
	\[
	\gamma^*_\cT\colon \SHS \leftrightarrows \cT\colon \gamma_*^\cT
	\]
	with $\gamma_\cT^*$ being strict monoidal.	We put
	\[
	\Ht_\cT\Z = \gamma_*^\cT \gamma^*_\cT \SSp\in \SHS
	\]
	and refer to it as \textit{the spectrum representing $\cT$-motivic cohomology}.
\end{definition}

\begin{example} \label{def:a1derived}
	The \textit{stable $\A^1$-derived category $\DAS$} is obtained from the category of complexes of Nisnevich sheaves of abelian groups on $\SmS$ inverting $\A^1$-weak equivalences and stabilizing with respect to the $\Gm$-suspension functor, see \cite[Chapter~5]{CD19} for the details. This category is tensor triangulated and admits a canonical adjunction of triangulated categories
		\[
		\homol_{\A^1}\colon\SHS \leftrightarrows \DAS \colon \emclane_{\A^1}
		\]
		with $\homol_{\A^1}$ being monoidal. The adjunction is the $\A^1$-derived $\Gm$-stable version of the classical adjunction 
		\[
		\homol_{\bullet}\colon\SH \leftrightarrows \mathrm{D}(\mathcal{A}b) \colon \emclane
		\]
		between the stable homotopy category and the derived category of abelian groups with $\homol_{\bullet}$ being the singular complex functor and $\emclane$ being the Eilenberg-Maclane functor.

		We put 
		\[
		\HAZ=\Ht_{\DAS} \Z =\emclane_{\A^1} \homol_{\A^1}\SSp\in \SHS
		\]
		and refer to it as \textit{the spectrum representing $\A^1$-cohomology}. See Lemma~\ref{lem:const_functor} for the identification of $\HAZ$ via the classical Eilenberg-Maclane spectrum $\Ht_{\mathrm{top}}\Z\in\SH$.
\end{example}

\begin{example} \label{ex:cor_cat}
	Let $S=\Spec k$ be the spectrum of a field. Many interesting triangulated motivic categories arise in the following way. Let $\mathrm{Cor}^\cT_k$ be a category of correspondences, i.e. a category with objects being those of $\Sm_k$ and the morphisms $\mathrm{Cor}^\cT_k(Y,X)$ given by formal linear combinations of some morphisms $Z\to Y\times X$ with possibly some additional data attached (e.g. a trivialization of $\omega_{Z/Y}$). The triangulated motivic category attached to $\mathrm{Cor}^\cT_k$ is the category of complexes of Nisnevich sheaves of abelian groups on $\mathrm{Cor}^\cT_k$ with inverted $\A^1$-weak equivalences and inverted $\Gm$-suspension functor. The adjunction is given by the compositions
	\[
	\gamma^*_\cT\colon \SHS \xleftrightarrows[\homol_{\A^1}]{\emclane_{\A^1}} \DAS \xleftrightarrows[\hphantom{x}\widetilde{\gamma}^*_\cT\hphantom{x}]{\widetilde{\gamma}_*^\cT} \cT \colon \gamma_*^\cT
	\]
	with the adjunction $\widetilde{\gamma}^*_\cT \dashv \widetilde{\gamma}_*^\cT$ induced by the graph functor $\gamma\colon \Sm_k \to \mathrm{Cor^\cT_k}$. The most interesting particular examples are the following ones:
	\begin{enumerate}
		\item Suppose that $k$ is perfect. Then we have the triangulated category of Voevodsky motives $\DMk$ based on the category of Voevodsky correspondences $\mathrm{Cor}_k$, see e.g. \cite{MVW06}. Here $\Ht_{\DMk}\Z=\HM\Z$ is the spectrum representing motivic cohomology.
		\item Suppose that $k$ is infinite, perfect and $\operatorname{char} k\neq 2$. Then we have the triangulated category of MW-motives $\DMWk$ based on the category of MW-correspondences $\widetilde{\mathrm{Cor}}_k$ \cite{BCDFO20}. Here $\Ht_{\DMWk}\Z=\HMW\Z$ is the spectrum representing MW-motivic cohomology.
		\item $\DM_A(k)$ based on the cohomological correspondence categories $\mathrm{Cor}^A_k$ introduced in \cite{DK18}. This example covers both $\DMk$ and $\DMWk$ discussed above (put $A=\HM \Z$ and $A=\HMW \Z$ respectively) but at the current moment this construction lacks the identification of $\Ht_{\DM_A(k)} \Z$ in terms of $A$. 
	\end{enumerate}
	Note that the basic constructions of the above categories do not need the base field to be infinite or perfect, but the comparison results identifying the cohomology theories represented by the spectra $\HM \Z$ and $\HMW \Z$ use in a crucial way the theory of presheaves with transfers which is developed for $\mathrm{Cor}_k$ over a perfect field and for $\widetilde{\mathrm{Cor}}_k $ over an infinite perfect field of $\operatorname{char}\neq 2$.
\end{example}

\begin{example} \label{ex:module_cat}
	Let $A\in \SHS$ be a highly structured commutative ring spectrum. Then we have the homotopy category of modules $\mathrm{Mod}_A$ which is tensor triangulated and comes with the adjunction
	\[
	\gamma^*_\cT\colon \SHS \leftrightarrows \mathrm{Mod}_A\colon \gamma_*^\cT.
	\]
	Here $\Ht_{\mathrm{Mod}_A}\Z=A$.
\end{example}

\begin{remark}
	Examples~\ref{ex:cor_cat} and~\ref{ex:module_cat} are closely related, see \cite[Theorem~1]{RO08}, \cite{EK17} and \cite[Theorem~5.3]{G19}.
\end{remark}

\begin{definition} \label{def:motive_Thom}
	Let $\cT$ be a triangulated motivic category over $S$ and consider the composition of functors
	\[
	\motive_\cT\colon \SmS \xrightarrow{\Sigma^\infty_\T(-)_+} \SHS \xrightarrow{\gamma_\cT^*} \cT.
	\]
	For $X\in \SmS$ we refer to 
	\[
	\motive_\cT(X)=\gamma_\cT^*(\Sigma^\infty_\T X_+) \in \cT
	\]
	as \textit{the $\cT$-motive of $X$}. For $X\in\SmS$ and an open embedding $j\colon U\to X$ we put 
	\[
	\motive_\cT(X/U)=\Cone\left( \motive_{\cT}(U) \xrightarrow{\motive_\cT(j)} \motive_{\cT}(X) \right).
	\]
	In particular we have
	\begin{itemize}
		\item the $\cT$-motive of the Thom space of a vector bundle $E$ over $X\in\SmS$ given by 
		\[
		\motive_\cT(\Th(E))=\Cone\left( \motive_{\cT}(E-z(X)) \xrightarrow{\motive_\cT(j)} \motive_{\cT}(E) \right)
		\]
		where $z\colon X\to E$ is the zero section and $j\colon E-z(X)\to E$ is the open embedding,		
		\item the $\cT$-motive of $\T$ given by 
		\[
		\motive_\cT(\T)=\Cone\left( \motive_{\cT}(\Gm) \xrightarrow{\motive_\cT(j)} \motive_{\cT}(\A^1) \right)
		\]
		where $j\colon \Gm\to \A^1$ is the open embedding.
	\end{itemize}
	For $\mc{N}\in\cT$ and $n\in \Z_{\ge 0}$ we put
	\[
	\Sigma_{\T}^n \mc{N} = \mc{N}\otimes \motive_{\cT}(\T)^{\otimes n}
	\]
	and refer to it as \textit{the $n$-fold $\T$-suspension of $\mc{N}$}. Lemma~\ref{lem:tmc_basic} below shows that the functor $\Sigma_\T$ is invertible allowing one to consider $\Sigma^n_\T$ for $n$ negative as well.

\end{definition}

\begin{lemma} \label{lem:tmc_basic}
	Let $\cT$ be a triangulated motivic category over $S$. Then
	\begin{enumerate}
		\item $\motive_\cT(\Th(E))=\gamma^*_\cT (\Sigma^\infty_\T \Th(E))$ for a vector bundle $E$ over $X\in\SmS$,
		\item $\motive_{\cT}(\T)= \gamma^*_\cT (\Sigma^\infty_\T \T)$, 
		\item $\motive_{\cT}(\T)$ is $\otimes$-invertible,
		\item $\Sigma^n_\T \gamma^*_\cT A = \gamma^*_\cT \Sigma^n_\T A$ for $A\in\SHS$ and $n\in \Z$.
		\item $\gamma_*^{\cT}\Sigma^n_\T \mc{N}=\Sigma^n_\T\gamma_*^{\cT}\mc{N}$ for $\mc{N}\in\cT$ and $n\in\Z$.
	\end{enumerate}
\end{lemma}
\begin{proof}
	Items (1) and (2) are straightforward.
	
	Item (3) follows from item (2) since $\gamma_\cT^*$ is monoidal and $\Sigma_{\T}^\infty \T$ is $\wedge$-invertible.
	
	Item (4) follows from item (2) in the following way:
	\[
	\Sigma^n_\T \gamma^*_\cT (A) = \gamma^*_\cT (A) \otimes \motive_{\cT}(\T)^{\otimes n} =  \gamma^*_\cT (A) \otimes  \gamma^*_\cT (\Sigma^\infty_\T \T)^{\otimes n} = \gamma^*_\cT (\Sigma_{\T}^n A).
	\]
	
	In order to prove item (5) take some $A\in\SHS$. Then applying adjunction, item (4) and the fact that $\Sigma_{\T}^n$ is invertible we have the following:
	\begin{multline*}
	\Hom_{\SHS}(A,\gamma_*^{\cT}\Sigma^n_\T \mc{N})=\Hom_{\cT}(\gamma^*_{\cT}A,\Sigma^n_\T \mc{N})
	=\Hom_{\cT}(\Sigma^{-n}_\T \gamma^*_{\cT}A,\mc{N})
	=\\
	=\Hom_{\cT}(\gamma^*_{\cT}\Sigma^{-n}_\T A,\mc{N})
	=\Hom_{\SHS}(\Sigma^{-n}_\T A,\gamma_*^{\cT}\mc{N})
	=\Hom_{\SHS}(A,\Sigma^{n}_\T\gamma_*^{\cT}\mc{N}).
	\end{multline*}
	This holds for every $A\in\SHS$ whence the claim.
\end{proof}

\begin{lemma} \label{lem:tmc_triangles}
	Let $\cT$ be a triangulated motivic category over $S$. Then
	\begin{enumerate}
		\item for an open cover $U\cup V=X\in\SmS$ and a vector bundle $E$ over $X$ there are distinguished triangles
		\begin{gather*}
		\motive_\cT(U\cap V) \xrightarrow{f} \motive_\cT(U) \oplus \motive_\cT(V)\xrightarrow{g} \motive_\cT(X) \to \motive_\cT(U\cap V) [1],\\
		\motive_\cT(\Th(E|_{U\cap V})) \xrightarrow{\widetilde{f}} \motive_\cT(\Th(E|_{U})) \oplus \motive_\cT(\Th(E|_{V}))\xrightarrow{\widetilde{g}} \motive_\cT(\Th(E)) \to \motive_\cT(\Th(E|_{U\cap V})) [1]
		\end{gather*}
		with the maps $f, g$ and $\widetilde{f}, \widetilde{g}$ induced by the respective open embeddings,
		\item for a closed embedding $Z\to X$ with $Z,X\in\SmS$ there is a distinguished triangle
		\[
		\motive_\cT(X-Z) \xrightarrow{\motive_\cT(j)} \motive_\cT(X)\to \motive_\cT(\Th(N_{Z/X})) \to \motive_\cT(X-Z) [1]
		\]
		with $j\colon X-Z\to X$ being the open embedding and $N_{Z/X}$ being the normal bundle.
	\end{enumerate}
\end{lemma}
\begin{proof}
	Follows from the corresponding sequences in $\SHS$ taking $\gamma^*_\cT$.
\end{proof}

\section{Cohomological and categorical Thom isomorphisms}

\begin{definition}[{see \cite[Section~2]{An20}}]
	A \textit{rank $n$ vector $G$-bundle $\mathcal{E}$ over $X\in \SmS$} with $G=\GL,\,\Sp,\,\SL$, or $\SLc$ is one of the following:
	\begin{itemize}
		\item $G=\GL\colon$ $\mathcal{E}=E$ a rank $n$ vector bundle over $X$,
		\item $G=\Sp\colon$ $\mathcal{E}=(E,\phi)$ where $E$ is a rank $n$ vector bundle over $X$ and $\phi$ is a symplectic form on $E$,
		\item $G=\SL\colon$ $\mathcal{E}=(E,\lambda)$ where $E$ is a rank $n$ vector bundle over $X$ and $\lambda\colon \det E\xrightarrow{\simeq} \triv_X$ is an isomorphism of line bundles,
		\item $G=\SLc\colon$ $\mathcal{E}=(E,L,\lambda)$ where $E$ is a rank $n$ vector bundle over $X$, $L$ is a line bundle over $X$ and $\lambda\colon \det E\xrightarrow{\simeq} L^{\otimes 2}$ is an isomorphism of line bundles.
	\end{itemize}
	A trivialized rank $n$ vector bundle over $X$ carries a canonical structure of a vector $G$-bundle (\cite[Section~2]{An20}). The corresponding vector $G$-bundle is denoted $\triv_X^{G,n}$ and referred to as the \textit{trivialized rank $n$ vector $G$-bundle over $X$}.
	
	In the above notation we put 
	\[
	\Th(\mcE)=\Th(E)=E/(E-z(X))
	\]
	with $z\colon X\to E$ being the zero section and refer to it as \textit{the Thom space of $\mcE$}.
\end{definition}

\begin{definition}[{cf. \cite[Definition~3.3]{An20}}]\label{def:coh_Thom}
	Let $G=\GL,\,\Sp,\,\SL$, or $\SLc$. We say that the cohomology theory represented by a commutative ring spectrum $A\in\SHS$ \textit{admits normalized Thom classes for vector $G$-bundles} if there is a rule which assigns to each rank $n$ vector $G$-bundle $\mathcal{E}$ over $X\in\SmS$ an element 
	\[
	\thc(\mathcal{E}) \in A^{2n,n}(\Th(\mathcal{E}))=\Hom_{\SHS}(\Sigma^\infty_\T\Th(\mathcal{E}),\Sigma^{2n,n} A)
	\]
	with the following properties:
	\begin{enumerate}
		\item For the trivialized rank $n$ vector $G$-bundle $\triv_X^{G,n}$ one has
		\[
		\thc(\triv_X^{G,n})=\Sigma^n_\T 1 \in A^{2n,n}(\Sigma^{n}_\T X_+)=A^{2n,n}(\Th(\triv_X^{G,n})).
		\]
		\item For an isomorphism $\theta\colon \mcE\xrightarrow{\simeq} \mcE'$ of vector $G$-bundles over $X\in\SmS$ one has 
		\[
		\thc(\mcE)=\widetilde{\theta}^A\thc(\mcE')
		\]
		where $\widetilde{\theta}\colon \Th(\mcE) \to \Th(\mcE')$ is the morphism induced by $\theta$.
		\item For a morphism $f\colon Y\to X$ in $\SmS$ and a vector $G$-bundle $\mcE$ over $X$ one has
		\[
		\widetilde{f}^A(\thc(\mcE))=\thc(f^*\mcE)
		\]
		where $\widetilde{f}\colon \Th(f^*\mcE)\to \Th(\mcE)$ is the morphism induced by the morphism of total spaces $f^*\mcE\to\mcE$.
	\end{enumerate}
	We refer to $\thc(\mcE)$ as \textit{Thom classes}. 
\end{definition}

\begin{remark} \label{rem:thom_iso}
	It is easy to see from the Mayer-Vietoris long exact sequence that for a cohomology theory with normalized Thom classes the homomorphism
	\[
	-\cup \thc(\mcE)\colon A^{*,*}(X)\to A^{*+2n,*+n}(\Th(\mcE))
	\]
	is an isomorphism, see e.g. \cite[Lemma~3.8]{An20}.
\end{remark}

\begin{remark}
	A commutative ring spectrum $A\in \SHS$ admits a \textit{normalized $G$-orientation} if in addition to the properties from Definition~\ref{def:coh_Thom} Thom classes are multiplicative with respect to direct sum of vector $G$-bundles (see \cite[Definition~3.3]{An20} for the details). This notion is closely related to the existence of a homomorphism of commutative ring spectra $MG\to A$, see \cite[Theorem~1.1]{PPR08}, \cite[Theorems~1.1,~5.9]{PW10}, \cite[Section~16]{BH17} and~\cite[Section~4.3]{BW20}.
\end{remark}

\begin{lemma} \label{lem:aut_gbundle}
	Let $G=GL,\,Sp,\,SL$, or $SL^c$ and $A\in\SHS$ be a commutative ring spectrum representing a cohomology theory that admits normalized Thom classes for vector $G$-bundles. Let $\theta\colon \mcE\to \mcE$ be an automorphism of a vector $G$-bundle $\mcE$ over $X\in\SmS$ and let $\widetilde{\theta}\colon \Th(\mcE)\to \Th(\mcE)$ be the corresponding automorphism of the Thom space. Then
	\[
	\widetilde{\theta}^A\colon A^{*,*}(\Th(\mcE)) \to A^{*,*}(\Th(\mcE))
	\]
	is the identity morphism.
\end{lemma}
\begin{proof}
	Remark~\ref{rem:thom_iso} yields that $A^{*,*}(\Th(\mcE))$ is a rank $1$ free $A^{*,*}(X)$-module generated by $\thc(\mcE)$. Property~$(2)$ of Definition~\ref{def:coh_Thom} yields that $\widetilde{\theta}^A(\thc(\mcE))=\thc(\mcE)$. The morphism $\widetilde{\theta}^A$ is a homomorphism of $A^{*,*}(X)$-modules whence the claim.
\end{proof}

\begin{definition} \label{def:thom_cat}
	Let $G=\GL,\,\Sp,\,\SL$, or $\SLc$. We say that a triangulated motivic category $\cT$ over $S$ \textit{admits normalized Thom isomorphisms for vector $G$-bundles} if there is a rule which assigns to each rank $n$ vector $G$-bundle $\mcE$ over $X\in \SmS$ an isomorphism
	\[
	\thc^\cT_\mcE\colon \motive_{\cT}(\Th(\mcE)) \xrightarrow{\simeq} \Sigma^n_\T\motive_{\cT}(X)
	\]
	with the following properties:
	\begin{enumerate}
		\item For the trivialized rank $n$ vector $G$-bundle $\triv_X^{G,n}$ the isomorphism
		\[
		\thc^\cT_{\triv_X^{G,n}}
		\colon \motive_{\cT}(\Th(\triv_X^{G,n}))\xrightarrow{\simeq}\Sigma^n_\T\motive_{\cT}(X)
		\]
		is the canonical one induced by the identification $\Th(\triv_X^{G,n})=\Th(\triv_X^{\oplus n})=\Sigma^n_\T X_+$.
		\item For an isomorphism $\theta\colon \mcE\xrightarrow{\simeq} \mcE'$ of vector $G$-bundles over $X\in\SmS$ one has 
		\[
		\thc^\cT_\mcE=\thc^\cT_{\mcE'} \circ \motive_\cT(\theta)
		\]
		where $\motive_\cT(\theta)\colon \motive_\cT(\Th(\mcE))\to \motive_\cT(\Th(\mcE'))$ is induced by the morphism of total spaces $\mcE\to \mcE'$ given by $\theta$.
		\item For a morphism $f\colon Y\to X$ in $\SmS$ and a vector $G$-bundle $\mcE$ over $X$ one has
		\[
		\motive_{\cT}(f)\circ \thc^\cT_{f^*\mcE} = \thc^\cT_\mcE \circ \motive_\cT(\widetilde{f})
		\]
		where $\motive_\cT(\widetilde{f})\colon \motive_\cT(\Th(f^*\mcE))\to \motive_\cT(\Th(\mcE))$ is induced by the morphism of total spaces $\widetilde{f}\colon f^*\mcE\to\mcE$.		
	\end{enumerate}
\end{definition}

\begin{lemma} \label{lem:thom_iso}
	Let $G=GL,\,Sp,\,SL$, or $SL^c$ and let $\cT$ be a triangulated motivic category over $S$. Suppose that there is a rule which assigns to each rank $n$ vector $G$-bundle $\mcE$ over $X\in \SmS$ a morphism
	\[
	\thc^\cT_\mcE\colon \motive_{\cT}(\Th(\mcE)) \xrightarrow{} \Sigma^n_\T\motive_{\cT}(X)
	\]
	satisfying properties (1), (2) and (3) of Definition~\ref{def:thom_cat}. Then $\thc^\cT_\mcE$ are isomorphisms.
\end{lemma}
\begin{proof}
	Follows from Lemma~\ref{lem:tmc_triangles}(1) similarly to the proof of \cite[Lemma~3.8]{An20}. 
\end{proof}

\begin{remark} \label{rem:Gysin}
	It follows from the homotopy purity property (Lemma~\ref{lem:tmc_triangles}(2)) that $\cT$ admits (possibly non-normalized and non-functorial) Thom isomorphisms for vector $G$-bundles if and only if for every closed embedding $Z\to X$ in $\SmS$ such that the normal bundle $N_{Z/X}$ admits a structure of a vector $G$-bundle there exists a distinguished triangle
	\[
	\motive_\cT(X-Z)\xrightarrow{\motive_\cT(j)} \motive_\cT(X) \to \Sigma^n_\T\motive_\cT(Z) \to \motive_\cT(X-Z)[1]
	\]
	where $j\colon X-Z\to X$ is the open embedding and $n$ is the codimension of $Z$ in $X$. These triangles are usually referred to as \textit{Gysin triangles}.
\end{remark}

\begin{theorem}\label{thm:thom_iso_equiv}
	Let $G=GL,\,Sp,\,SL$, or $SL^c$ and let $\cT$ be a triangulated motivic category over $S$. Then the following are equivalent:
	\begin{enumerate}
		\item $\cT$ admits normalized Thom isomorphisms for vector $G$-bundles,
		\item the cohomology theory represented by the spectrum $\Ht_\cT \Z\in \SHS$ admits normalized Thom classes for vector $G$-bundles.
	\end{enumerate}
\end{theorem}
\begin{proof}
	$(1) \Rightarrow (2)$. Let $\mcE$ be a rank $n$ vector $G$-bundle over $X\in\SmS$. Adjunction combined with Lemma~\ref{lem:tmc_basic} yields a bijection
	\begin{multline*}
	\Hom_{\cT}(\motive_\cT(\Th(\mcE)),\Sigma^n_\T \motive_\cT (X)) \xrightarrow{\simeq}\Hom_{\SHS}(\Sigma^\infty_\T\Th(\mcE),\gamma^\cT_*\Sigma^n_\T \motive_\cT (X))
	\xrightarrow{\simeq}\\
	\xrightarrow{\simeq}
	\Hom_{\SHS}(\Sigma^\infty_\T\Th(\mcE),\Sigma^n_\T \gamma^\cT_* \gamma_\cT^* (\Sigma^\infty_\T X_+))	
	\end{multline*}
	which we denote by $\Phi$. The structure morphism $X\to S$ yields a homomorphism
	\[
	p\colon \Hom_{\SHS}(\Sigma^\infty_\T\Th(\mcE),\Sigma^n_\T\gamma^\cT_* \gamma_\cT^*  (\Sigma^\infty_\T X_+))\to
	\Hom_{\SHS}(\Sigma^\infty_\T\Th(\mcE),\Sigma^n_\T \gamma^\cT_* \gamma_\cT^* (\Sigma^\infty_\T S_+)).
	\]
	Note that $\Sigma^n_\T \gamma^\cT_* \gamma_\cT^* (\Sigma^\infty_\T S_+) = \Sigma^n_\T \gamma^\cT_* \gamma_\cT^* \SSp = \Sigma^n_\T\Ht_{\cT}\Z$ whence we get a homomorphism
	\[
	p\circ \Phi\colon
	\Hom_{\cT}(\motive_\cT(\Th(\mcE)),\Sigma^n_\T \motive_\cT (X)) \to
	(\Ht_{\cT}\Z)^{2n,n}(\Th(\mcE)).
	\]
	Put
	\[
	\thc(\mcE)=p \circ \Phi (\thc^{\cT}_\mcE)\in (\Ht_{\cT}\Z)^{2n,n}(\Th(\mcE)).
	\]
	All the properties from Definition~\ref{def:coh_Thom} are straightforward to check using the corresponding properties from Definition~\ref{def:thom_cat}.
	
	$(2) \Rightarrow (1)$.	Let $\mcE$ be a rank $n$ vector $G$-bundle over $X\in\SmS$. Denote $\Psi$ the bijection
	\[
	(\Ht_{\cT}\Z)^{2n,n}(\Th(\mcE))=\Hom_{\SHS}(\Sigma^\infty_\T \Th(\mcE), \Sigma_\T^n \gamma_*^\cT\gamma^*_\cT \SSp) \xrightarrow{\simeq}
	\Hom_{\cT}(\motive_{\cT}(\Th(\mcE)), \Sigma_\T^n\motive_\cT(S))
	\]
	given by adjunction and Lemma~\ref{lem:tmc_basic} and let
	\[
	\Delta \colon \motive_{\cT}(\Th(\mcE))\to \motive_{\cT}(X) \otimes \motive_{\cT}(\Th(\mcE))
	\]
	be the morphism induced by the diagonal
	\[
	\Sigma^\infty_\T \Th(\mcE)\to \Sigma^\infty_\T X_+ \wedge \Sigma^\infty_\T \Th(\mcE).
	\]
	Put $\thc^{\cT}_{\mcE} =  (\id_{\motive_{\cT}(X)} \otimes \Psi(\thc(\mcE))) \circ \Delta$:
	\[
	\motive_{\cT}(\Th(\mcE))\xrightarrow{\Delta} \motive_{\cT}(X) \otimes \motive_{\cT}(\Th(\mcE)) \xrightarrow{\id_{\motive_{\cT}(X)} \otimes \Psi(\thc(\mcE))} \motive_{\cT}(X) \otimes \Sigma^n_\T\motive_{\cT}(S) = \Sigma^n_\T\motive_{\cT}(X).
	\]
	All the properties from Definition~\ref{def:thom_cat} are straightforward to check using the corresponding properties from Definition~\ref{def:coh_Thom}. Then $\thc^{\cT}_{\mcE}$ is an isomorphism by Lemma~\ref{lem:thom_iso} whence the claim.	
\end{proof}

\begin{corollary} \label{cor:cat_orient_sheaf}
	Let $\cT$ be a triangulated motivic category over $S$ and suppose that the presheaf $(\Ht_{\cT} \Z)^{0,0}(-)$ is a Zariski sheaf on $\SmS$. Then $\cT$ admits normalized Thom isomorphisms for vector $SL^c$-bundles.
\end{corollary}
\begin{proof}
	Follows from \ref{thm:thom_iso_equiv} and \cite[Theorem~5.3]{An20}. 
\end{proof}

\begin{remark}
	Suppose that $S=\Spec k$ is the spectrum of an infinite perfect field of $\operatorname{char}\neq 2$. Then it follows from \cite[Chapter~3, Theorem~4.2.4]{BCDFO20} that $(\HMW\Z)^{0,0}(-)$ is given by the unramified Grothendieck-Witt sheaf of bilinear forms and Corollary~\ref{cor:cat_orient_sheaf} yields the existence of normalized Thom isomorphisms
	\[
	\motive_{\DMW(k)} (\Th(\mcE)) \xrightarrow{\simeq} \Sigma^{n}_\T\motive_{\DMW(k)} (X)
	\]
	in $\DMW(k)$ for vector $\SLc$-bundles recovering \cite[Theorem~6.2]{Y19}.
	
	In a similar way for $S=\Spec k$ being the spectrum of perfect field Theorem~\ref{thm:thom_iso_equiv} combined with the orientation of motivic cohomology yields the existence of normalized Thom isomorphisms 
	\[
	\motive_{\DM(k)} (\Th(E)) \xrightarrow{\simeq} \Sigma^{n}_\T\motive_{\DM(k)} (X)
	\]
	in $\DM(k)$ for vector bundles recovering \cite[Theorem~15.15]{MVW06}.
\end{remark}

\section{First homology of the motivic sphere spectrum} \label{section:homology}

\begin{definition} \label{def:constfunctor}
	The \textit{constant functor} $\constf\colon \SH\to \SHS$ is the functor given by the composition
	\[
	\SH \xrightarrow{c} \mathcal{SH}_{S^1}(S) \xrightarrow{\Sigma_{\Gm}^\infty} \SHS
	\]
	where $c$ takes an $\mathrm{S^1}$-spectrum of simplicial sets to the corresponding $\mathrm{S^1}$-spectrum of constant  presheaves of simplicial sets and $\Sigma_{\Gm}^\infty$ is the suspension functor. Since both $c$ and $\Sigma_{\Gm}^\infty$ are triangulated and monoidal the functor $\constf$ is triangulated and monoidal as well.
	
	The classical sphere
spectrum and Eilenberg-MacLane spectrum in the stable homotopy category $\SH$ are denoted $\SSp_\mathrm{top}$ and $\Ht_\mathrm{top}\Z$, respectively.
\end{definition}

\begin{lemma} \label{lem:const_functor}
	There are canonical isomorphisms 
	\[
	\SSp = \constf(\SSp_{\mathrm{top}}),\quad \ZS = \constf(\ZStop)
	\]
	and $\constf$ takes the unit morphism $\SSp_{\mathrm{top}}\to \ZStop$ to the unit morphism $\SSp\to \ZS$.
\end{lemma}
\begin{proof}
	The first equality is essentially the definition of $\SSp$.
	
	For the second equality consider the following diagrams.
	\[
	\xymatrix{
	\SHS \ar[r]^{\homol_{\A^1}} & \DAS \\
	\SH \ar[u]^{\constf} \ar[r]^{\homol_{\bullet}} & \mathrm{D}(\mathcal{A}b) \ar[u]_{\constf_\Z}
	}
	\quad 
	\xymatrix{
	\SHS  & \DAS \ar[l]_{\emclane_{\A^1}} \\
	\SH \ar[u]^{\constf}  & \mathrm{D}(\mathcal{A}b) \ar[u]_{\constf_\Z} \ar[l]_{\emclane}
	}	
	\]
	Here ${\constf}_\Z$ is the linear version of ${\constf}$ while the other functors are from Example~\ref{def:a1derived}. We claim that these diagrams commute. Indeed, these diagrams are derived (and $\Gm$-stabilized in the top row) versions of the diagrams
	\[
	\xymatrix{
	\Spt_{S^1}(S) \ar[r]^{\underline{\homol}_{\A^1}} & \mathrm{Ch}_{\bullet}(S) \\
	\Spt_{S^1} \ar[u]^{\underline{\constf}} \ar[r]^{\underline{\homol}_{\bullet}} & \mathrm{Ch}_{\bullet}(\mathcal{A}b) \ar[u]_{\underline{\constf}_\Z}
	}
	\quad 
	\xymatrix{
	\Spt_{S^1}(S)  & \mathrm{Ch}_{\bullet}(S) \ar[l]_(0.4){\underline{\emclane}_{\A^1}} \\
	\Spt_{S^1} \ar[u]^{\underline{\constf}}  & \mathrm{Ch}_{\bullet}(\mathcal{A}b) \ar[u]_{\underline{\constf}_\Z} \ar[l]_{\underline{\emclane}}
	}	
	\]
	where
	\begin{itemize}
		\item $\Spt_{S^1}$ is the classical category of $S^1$-spectra of simplicial sets,
		\item $\mathrm{Ch}_{\bullet}(\mathcal{A}b)$ is the category of chain complexes of abelian groups,
		\item $\Spt_{S^1}(S)$ is the category of $S^1$-spectra of simplicial sheaves,
		\item $\mathrm{Ch}_{\bullet}(S)$ is the category of chain complexes of abelian group sheaves,
		\item ${\underline{\constf}}$ and ${\underline{\constf}}_\Z$ are the constant sheaf functors,
		\item ${\underline{\homol}_\bullet}$ and $\underline{\emclane}$ are the classical singular complex and Eilenberg-Maclane functors,
		\item ${\underline{\homol}_{\A^1}}$ and $\underline{\emclane}_{\A^1}$ are sheaf versions of the classical singular complex and Eilenberg-Maclane functors, i.e. sheafifications of the presheaves given sectionwise by ${\underline{\homol}_\bullet}$ and $\underline{\emclane}$.
	\end{itemize}
	The non-derived diagrams clearly commute. In the injective model structures the functors $\underline{\constf},\, \underline{\constf}_\Z,\, \underline{\homol}_{\bullet},\, \underline{\homol}_{\A^1}$ are left Quillen, whence $\homol_{\A^1} \circ \constf = \constf_\Z  \circ  \homol_{\bullet}$. Recall that constant sheaves have trivial higher Nisnevich cohomology groups whence in the projective model structures the functors $\underline{\constf}$ and $\underline{\constf}_\Z$ take fibrant objects to motivically fibrant objects. Then $\underline{\constf},\, \underline{\constf}_\Z,\, \underline{\emclane},\, \underline{\emclane}_{\A^1}$ are right Quillen in the respective projective model structures whence  $\constf \circ \emclane =  \emclane_{\A^1} \circ \constf_\Z$.

	From the commutativity of the diagrams we have
	\begin{multline*}
	\ZS = \emclane_{\A^1} (\homol_{\A^1} \SSp) = \emclane_{\A^1} (\homol_{\A^1} (\constf(\SSp_{\mathrm{top}}))) = \\
	= \emclane_{\A^1} (\constf_\Z ( \homol_{\bullet} (\SSp_{\mathrm{top}}))) = \constf (  \emclane (\homol_{\bullet} (\SSp_{\mathrm{top}}))) = \constf(\ZStop).
	\end{multline*}
	The last assertion of the lemma also follows from the commutativity of the above diagrams.
\end{proof}

\begin{definition}
	Let $S=\Spec k$ be the spectrum of a field. For $A\in\SHk$ and $i,n\in\Z$ we denote $\underline{\pi}_i(A)_n$ the Nisnevich sheaf on $\Smk$ associated with the presheaf
	\[
	U\mapsto \Hom_{\SHk}(\Sigma^{i-n,-n}\Sigma^\infty_\T U_+, A),\quad U\in\Smk.
	\]
	We refer to $\underline{\pi}_i(A)_n$ as \textit{homotopy sheaves}. Following \cite[Section 5.2]{Mor04} for $j\in\Z$ we put
	\[
	\SHk_{\ge j}=\{A\in \SHk\,|\, \underline{\pi}_{i}(A)_n=0\, \text{for $i<j$, $n\in \Z$}\}
	\]
	for the full subcategory of $\SHk$ consisting of $(j-1)$-connected spectra.
	
	Similarly for $j\in\Z$ we have 
	\[
	\SH_{\ge j}=\{A\in \SH\,|\, \pi_{i}(A)=0\, \text{for $i<j$}\}
	\]	
	the full subcategory of $\SH$ consisting of $(j-1)$-connected spectra.
\end{definition}

\begin{lemma} \label{lem:connectivity}
	Let $S=\Spec k$ be the spectrum of a field. Then for $j\in \Z$ we have 
	\[
	\constf (\SH_{\ge j})\subset \SH(k)_{\ge j}.
	\]
\end{lemma}
\begin{proof}
	Follows from \cite[Theorem~4.2.10]{Mor04}.
\end{proof}

\begin{definition}
	For the Hopf element $\eta_{\mathrm{top}}\in \pi_1(\SSp_{\mathrm{top}})$ we put 
	\[
	\eta_{\mathrm{top}}^{\A^1}=\constf(\eta_{\mathrm{top}})\in \underline{\pi}_1(\SSp)_0(S).
	\]
\end{definition}

\begin{theorem} \label{thm:a1homology}
	Let $S=\Spec k$ be the spectrum of a field. Then the following sequence is exact:
	\[
	\underline{\pi}_0(\SSp)_*\xrightarrow{\cup\eta^{\A^1}_{\mathrm{top}}} \underline{\pi}_1(\SSp)_* \xrightarrow{H} \underline{\pi}_1(\ZS)_*\to 0.
	\]
	Here $H$ is induced by the unit morphism $\SSp\to \ZS$.
\end{theorem}
\begin{proof}
	Consider the following diagram in $\SH$.
	\[
	\xymatrix{
	 & & \Cone(f)[-1] \ar[d] \\ 
	 \SSp_{\mathrm{top}}[1] \ar[r]^{\eta_{\mathrm{top}}}  & \SSp_{\mathrm{top}} \ar[dr]_{u} \ar[r] & \Cone(\eta_{\mathrm{top}}) \ar[d]^f \\
	 & & \HZtop
	}
	\]
	Here $u$ is the unit morphism and $f$ is induced by the equality $u\circ \etatop=0$. Recall that $\pi_1(\SSp_{\mathrm{top}})$ is generated by $\etatop$ whence from the long exact sequence of homotopy groups of the horizontal row we get $\pi_1(\Cone(\etatop))=0$. Moreover, $\pi_0(\Cone(\etatop))=\Z$ and
	\[
	\pi_0(f)\colon \pi_0(\Cone(\etatop)) \to \pi_0(\HZtop)
	\]
	is an isomorphism. Thus $\Cone(f)[-1]\in \SH_{\ge 2}$.
	
	Applying $\constf$ to the above diagram and using Lemma~\ref{lem:const_functor} we obtain
	\[
	\xymatrix{
		& & \constf(Cone(f)[-1]) \ar[d] \\ 
		\SSp[1] \ar[r]^{\eta^{\A^1}_{\mathrm{top}}}  & \SSp \ar[dr] \ar[r] & \constf(\Cone(\eta_{\mathrm{top}})) \ar[d]^{\constf(f)} \\
		& & \ZS
	}
	\]	
	with the diagonal morphism being the unit morphism. Lemma~\ref{lem:connectivity} yields that 
	\[
	\constf(\Cone(f)[-1])\in \SHk_{\ge 2}
	\]
	whence it follows from the long exact sequence of homotopy sheaves of the column that
	\[
	\underline{\pi}_1(\constf(f))_*\colon \underline{\pi}_1(\constf(\Cone(\eta_{\mathrm{top}})))_* \to \underline{\pi}_1(\ZS)_*
	\]
	is an isomorphism. The claim follows from the long exact sequence of homotopy sheaves of the row since $\underline{\pi}_0(\SSp[1])_*=0$.
\end{proof}

\begin{remark}
	Theorem~\ref{thm:a1homology} allows one to compute $\underline{\pi}_1(\ZS)_*$ out of the description of $\underline{\pi}_1(\SSp)_*$ given by \cite[Theorem~5.5]{RSO19} and the computation of $\underline{\pi}_0(\SSp)_*$ given by \cite[Corollary~1.25]{Mor12}, see e.g. Corollary~\ref{cor:h1weightm2} below.
\end{remark}

\begin{definition} \label{def:nu}
	Following \cite[Definition~4.7, Remark~4.4]{DI13} and \cite[p190]{Mor12} we denote $\nu\in \underline{\pi}_1(\SSp)_{-2}(S)$ the second motivic Hopf map given by the stabilization of the Hopf construction for $\SL_2=(\SL_2,1)$,
	\[
	S^{7,4}\cong \Sigma^{1,0}(\SL_2 \wedge \SL_2) \xleftarrow{\simeq} \SL_2 * \SL_2 \to \Sigma^{1,0}(\SL_2\times \SL_2) \xrightarrow{\Sigma^{1,0}m} \Sigma^{1,0}\SL_2 \cong S^{4,2},
	\]
	where $m(x,y)=xy$ is the group multiplication.
\end{definition}

\begin{corollary} \label{cor:h1weightm2}
	Let $S=\Spec k$ be the spectrum of a field of exponential characteristic $p\neq 2$. 
	\begin{enumerate}
	\item
	If $p\neq 3$ then there is an isomorphism $\underline{\pi}_1(\ZS)_{-2}(k)[\tfrac{1}{p}]\cong \Z/12\Z$ with a generator given by $H(\nu)$,
	\item
	If $p =3$ then there is an isomorphism $\underline{\pi}_1(\ZS)_{-2}(k)[\tfrac{1}{3}]\cong \Z/4\Z$ with a generator given by $H(\nu)$.
	\end{enumerate}
\end{corollary}
\begin{proof}
	\cite[Theorem~5.5, Remark~5.8]{RSO19} yield that
	\[
	\underline{\pi}_1(\SSp)_{-2}(k)[\tfrac{1}{p}]\cong \Z/24\Z[\tfrac{1}{p}]
	\]
	with $\nu$ being a generator. Consider the exact sequence 
	\[
	\underline{\pi}_0(\SSp)_{-2}(k)[\tfrac{1}{p}] \xrightarrow{\cup\eta_{\mathrm{top}}^{\A^1}} \underline{\pi}_1(\SSp)_{-2}(k)[\tfrac{1}{p}] \xrightarrow{H} \underline{\pi}_1(\ZS)_{-2}(k)[\tfrac{1}{p}]\to 0.
	\]
	from Theorem~\ref{thm:a1homology}. The image of $\cup\eta_{\mathrm{top}}^{\A^1}$
	is given by 
	\[
	\mathrm{K^{MW}_{-2}}(k)[\tfrac{1}{p}]\cdot \eta_{\mathrm{top}}^{\A^1}=\GW(k)[\tfrac{1}{p}]\cdot\eta^2\eta_{\mathrm{top}}^{\A^1}.
	\]
	By \cite[Remark~5.8]{RSO19} we have $\eta^2\eta_{\mathrm{top}}^{\A^1}=12\nu$ in 	$\underline{\pi}_1(\SSp)_{-2}(k)[\tfrac{1}{p}]$ whence 
	\[
	\GW(k)[\tfrac{1}{p}]\cdot\eta^2\eta_{\mathrm{top}}^{\A^1}=\GW(k)[\tfrac{1}{p}]\cdot12\nu.
	\]
	For the fundamental ideal $\mathrm{I}(k)\subset \GW(k)$ we have 
	\[
	\mathrm{I}(k)\cdot 12\nu =\mathrm{K^{MW}_1}(k)\cdot  12\eta\nu = 0
	\]
	by \cite[Theorem~1.4]{DI13} whence the image of $\cup\eta_{\mathrm{top}}^{\A^1}$ is generated by $12\nu$ and the claim follows.
\end{proof}

\begin{lemma}\label{lem:nuh1}
	The element $H(\nu)\in \underline{\pi}_1(\ZS)_{-2}(S)$ is not zero.
\end{lemma}
\begin{proof}
	Via base change to a geometric point of $S$ we may assume that $S=\Spec k$ is the spectrum of an algebraically closed field. If $\operatorname{char} k\neq 2$ then the claim follows from Corollary~\ref{cor:h1weightm2}. 
	
	Suppose that $\operatorname{char} k = 2$. Applying Theorem~\ref{thm:a1homology} it is sufficient to show that
	\[
	\nu\notin \mathrm{K^{MW}_{-2}}(k)\cdot \eta_{\mathrm{top}}^{\A^1}\subset \underline{\pi}_1(\SSp)_{-2}(k).
	\]
	Since $k$ is algebraically closed $\mathrm{K^{MW}_{-2}}(k)\cong \Z/2\Z$ thus it is sufficient to show that $2\nu\neq 0$. Put $l=3$ and consider the $l$-complete \etale realization functors $\mathrm{L\hat{E}t}$ of \cite{I04,Q07}. Arguing as in the proof of \cite[Lemma~3.5]{AFW20} one obtains $\mathrm{L\hat{E}t}(\SL_2)\cong (\mathrm{S}^3)^{\wedge}_3$.
	The element $\nu$ is given by the Hopf construction whence $\mathrm{L\hat{E}t}(\nu)=\nu_{\mathrm{top}}$
	for the corresponding element $\nu_{\mathrm{top}}\in\pi_3(\SSp)^{\wedge}_3$ arising from the Hopf invariant one map (see \cite[page 848]{DI13}). Thus $\mathrm{L\hat{E}t}(\nu)\in \pi_3(\SSp)^{\wedge}_3\cong \Z/3\Z$ is a generator and, in particular, $2\mathrm{L\hat{E}t}(\nu)\neq 0$ whence $2\nu\neq 0$.
\end{proof}

\section{Thom isomorphisms and vanishing of $\nu$}

\begin{lemma} \label{lem:nu_as_representation}
	Let $p_+\colon (\SL_2)_+\to S_+$ be the structure morphism and let $q\colon (\SL_2)_+\to (\SL_2,1)$ be the morphism identifying $+$ with $1$. Then under the identification 
	\[
	\pi_{1}(\SSp)_{-2}(S)\cong \Hom_{\SHS}(\Sigma^\infty_\T (\SL_2,1), \SSp)
	\]
	one has
	\[
	\Sigma^2_\T (\nu \circ \Sigma^\infty_\T q) = \Sigma^\infty_\T\widetilde{\nu} - \Sigma^2_\T\Sigma^\infty_\T p_+ \in \Hom_{\SHS}(\Sigma^\infty_\T (\SL_2)_+\wedge \T^{\wedge 2}, \Sigma^\infty_\T \T^{\wedge 2})
	\]
	where $\widetilde{\nu}\colon (\SL_2)_+\wedge \T^{\wedge 2} \to \T^{\wedge 2}$ is the morphism induced by the canonical action $\SL_2\times \A^2\to \A^2$. Moreover, $\nu$ is the unique element of $\pi_{1}(\SSp)_{-2}(S)$ satisfying this equality.
\end{lemma}
\begin{proof}
	The uniqueness follows since
	\[
	(\Sigma^\infty_\T q, \Sigma^\infty_\T p_+)\colon \Sigma^\infty_\T (\SL_2)_+ \to \Sigma^\infty_\T (\SL_2,1) \oplus \SSp
	\]
	is an isomorphism, whence 
	\[
	\Sigma^2_\T (- \circ \Sigma^\infty_\T q)\colon \Hom_{\SHS}(\Sigma^\infty_\T (\SL_2,1), \SSp) \to \Hom_{\SHS}(\Sigma^\infty_\T (\SL_2)_+\wedge \T^{\wedge 2}, \Sigma^\infty_\T \T^{\wedge 2})
	\]
	is injective.
	
	In order to prove the equality consider the following diagram.
	\[
	\xymatrix{
	\Sigma^{1,0}(\SL_2\wedge \SL_2) \ar[d]_{\simeq} & \SL_2* \SL_2 \ar[l]_(0.4){\simeq} \ar[r] \ar[d]_{\simeq} & \Sigma^{1,0}(\SL_2 \times \SL_2) \ar[r] \ar[d]_{\simeq} & \Sigma^{1,0}\SL_2 \ar[d]_{\simeq} \\
	\Sigma^{1,0}(\SL_2\wedge (\A^2-0)) \ar[d]_{\simeq} & \SL_2* (\A^2-0) \ar[l]_(0.4){\simeq} \ar[r] & \Sigma^{1,0}(\SL_2 \times (\A^2-0)) \ar[r] \ar[d]_{\phi}  & \Sigma^{1,0}(\A^2-0) \ar[d]_{=} \\
	\SL_2\wedge \Sigma^{1,0}(\A^2-0) \ar[d]_{\id\wedge \psi}^\simeq & &  (\SL_2)_+\wedge \Sigma^{1,0}(\A^2-0) \ar[r] \ar[d]_{\id \wedge \psi}^\simeq & \Sigma^{1,0}(\A^2-0) \ar[d]_{\psi}^\simeq\\	
	\SL_2\wedge \T^{\wedge 2}& &  (\SL_2)_+\wedge \T^{\wedge 2} \ar[r]^(0.6){\widetilde{\nu}} \ar[ll]_{q\wedge \id} & \T^{\wedge 2} \\	
	}
	\]
	Here
	\begin{itemize}
		\item 
		$\SL_2=(\SL_2,1)$, $\A^2-0 = (\A^2-0, (1,0))$,
		\item 
		the first row is the one from Definition~\ref{def:nu},
		\item
		the second row is the first one with one copy of $\SL_2$ changed to $\A^{2}-0$ and with the last morphism being the $\Sigma^{1,0}$-suspension of the canonical action $\SL_2\times(\A^2-0)\to \A^2-0$,
		\item 
		the morphisms between the first and the second rows are given by the projection to the first column $\SL_2\to \A^2-0$ which is an $\A^1$-equivalence (the morphism is a Zariski locally trivial fibration with the fiber being $\A^1$),
		\item 
		$\phi$ is the quotient morphism
		\[
		\xymatrix{
		\Sigma^{1,0}(\SL_2 \times (\A^2-0))	\ar[d]^= & (\SL_2)_+\wedge \Sigma^{1,0}(\A^2-0)) \ar[d]^=\\
		\frac{\SL_2\times (\A^2-0) \times I}{\left(\SL_2\times (\A^2-0) \times \partial I\right)\cup (1\times (1,0) \times I)} \ar[r] &
		\frac{\SL_2\times (\A^2-0) \times I}{\left(\SL_2\times (\A^2-0) \times \partial I\right)\cup (\SL_2\times (1,0) \times I)}
		}
		\]
		with $I$ being the simplicial interval,
		\item 
		$\psi\colon \Sigma^{1,0}(\A^2-0)\to \T^{\wedge 2}$ is the canonical $\A^1$-equivalence that agrees with the $\SL_2$-action.
	\end{itemize}
	It is straightforward to check that the diagram commutes. Let
	\[
	\chi = (\id \wedge \psi)\circ \phi \circ \hdots \circ (\id\wedge \psi)^{-1}\colon \SL_2\wedge \T^{\wedge 2} \to (\SL_2)_+\wedge \T^{\wedge 2}
	\]
	be the morphism given by the clockwise composition in the bottom-left square of the diagram (taking inverses of the corresponding isomorphisms if needed) and let $i\colon \T^{\wedge 2} \to (\SL_2)_+\wedge \T^{\wedge 2}$ be the morphism induced by the inclusion $1\to \SL_2$. Following \cite[Appendix~A and~B]{DI13} one obtains that 
	\[
	\Sigma^\infty_\T (\chi, i)\colon \Sigma^\infty_\T \SL_2\wedge \T^{\wedge 2} \oplus \Sigma^\infty_\T \T^{\wedge 2} \to \Sigma^\infty_\T (\SL_2)_+\wedge \T^{\wedge 2}
	\]
	is the isomorphism inverse to
	\[
	\Sigma^\infty_\T (\Sigma^2_\T q, \Sigma^2_\T p_+)\colon \Sigma^\infty_\T (\SL_2)_+\wedge \T^{\wedge 2} \to \Sigma^\infty_\T \SL_2\wedge \T^{\wedge 2} \oplus \Sigma^\infty_\T \T^{\wedge 2}.
	\]
	This combined with the commutativity of the above diagram yields
	\begin{gather*}
	\Sigma^2_\T (\nu \circ \Sigma^\infty_\T q) \circ \Sigma^\infty_\T \chi = \Sigma^2_\T \nu = \Sigma^\infty_\T (\widetilde{\nu} \circ \chi) = \Sigma^\infty_\T\widetilde{\nu} \circ \Sigma^\infty_\T\chi - \Sigma^2_\T\Sigma^\infty_\T p_+ \circ \Sigma^\infty_\T \chi,\\
	\Sigma^2_\T (\nu \circ \Sigma^\infty_\T q) \circ \Sigma^\infty_\T i = \Sigma^2_\T \nu \circ \Sigma^\infty_\T \Sigma^2_\T q \circ \Sigma^\infty_\T i = 0, \\ 
	(\Sigma^\infty_\T \widetilde{\nu} - \Sigma^2_\T\Sigma^\infty_\T p_+) \circ \Sigma^\infty_\T i =\Sigma^\infty_\T \id_{\T^{\wedge 2}} - \Sigma^\infty_\T \id_{\T^{\wedge 2}} =0.
	\end{gather*}
	Thus the necessary equality holds composed with $(\Sigma^\infty_\T \chi, \Sigma^\infty_\T i)$. Since $(\Sigma^\infty_\T \chi, \Sigma^\infty_\T i)$ is an isomorphism the claim follows.
\end{proof}

\begin{remark}
	Roughly speaking, Lemma~\ref{lem:nu_as_representation} says that up to suspension and canonical isomorphisms between motivic spheres the Hopf map $\nu$ is the action $(\SL_2)_+\wedge \T^{\wedge 2} \to \T^{\wedge 2}$ projected to the summand $(\SL_2,1)$ in the decomposition $\Sigma^\infty_\T (\SL_2)_+ = \Sigma^\infty_\T (\SL_2,1) \oplus \SSp$.
\end{remark}

\begin{lemma}\label{lem:nusymp}
	Let $A\in \SHS$ be a commutative ring spectrum. Suppose that the cohomology theory represented by $A$ admits normalized Thom classes for symplectic bundles (i.e. vector $Sp$-bundles in the notation of Definition~\ref{def:coh_Thom}). Then $A(\nu)=0$ where $A(\nu)\in \pi_1(A)_{-2}(S)$ is the image of $\nu$ under the homomorphism
\[
\pi_1(\SSp)_{-2}(S)\to \pi_1(A)_{-2}(S)
\]
induced by the unit morphism $\SSp\to A$.
\end{lemma}
\begin{proof}
	In view of Lemma~\ref{lem:nu_as_representation}	it is sufficient to show that
	\[
	\widetilde{\nu}^A(\Sigma^2_\T 1) = (\Sigma^2_\T p_+)^A(\Sigma^2_\T 1) \in A^{4,2}((\SL_2)_+\wedge \T^{\wedge 2}).
	\]
	in the notation of Lemma~\ref{lem:nu_as_representation}. Note that $(\SL_2)_+\wedge \T^{\wedge 2} = \Th(\triv^{\oplus 2}_{\SL_2})$ and consider the following diagram.
	\[
	\xymatrix{
	\Th(\triv^{\oplus 2}_{\SL_2}) \ar[r]^(0.45)= \ar[d]^{\widetilde{\theta}} &  	(\SL_2)_+\wedge \T^{\wedge 2} \ar[r]^(0.65){\widetilde{\nu}} & \T^{\wedge 2} \ar[d]^= \\
	\Th(\triv^{\oplus 2}_{\SL_2}) \ar[r]^(0.45)= &  	(\SL_2)_+\wedge \T^{\wedge 2} \ar[r]^(0.65){\Sigma^2_\T p_+} & \T^{\wedge 2} 
	}
	\]
	Here $\widetilde{\theta}$ is induced by the morphism 
	\[
	\theta\colon \SL_2\times \A^2\to \SL_2\times \A^2, \quad \theta(g,v)=(g,gv).
	\]
	The diagram clearly commutes. The morphism $\theta$ may be regarded as an automorphism of the trivialized rank $2$ symplectic bundle $\triv^{\Sp,2}_{\SL_2}$ over $\SL_2$ thus applying Lemma~\ref{lem:aut_gbundle} we get $\widetilde{\theta}^A=\id$ and
	\[
	\widetilde{\nu}^A(\Sigma^2_\T 1) = \widetilde{\theta}^A (\Sigma^2_\T p_+)^A (\Sigma^2_\T 1)  = (\Sigma^2_\T p_+)^A(\Sigma^2_\T 1). \qedhere 
	\]
\end{proof}

\begin{remark}
	The above lemma is a symplectic counterpart of the well known statement that if a commutative ring spectrum $A\in\SHS$ represents a cohomology theory admitting Thom isomorphisms for arbitrary vector bundles then $A(\eta)=0$.
\end{remark}

\begin{theorem} \label{thm:DAS_Thom}
	Let $G=GL,\,Sp,\,SL$, or $SL^c$. Then 
	\begin{enumerate}
		\item 
		the cohomology theory represented by $\ZS$ does not admit normalized Thom classes for vector $G$-bundles,
		\item 
		$\DAS$ does not admit normalized Thom isomorphisms for vector $G$-bundles.
	\end{enumerate}		
\end{theorem}
\begin{proof}
	The first item follows from Lemmas~\ref{lem:nuh1} and~\ref{lem:nusymp}, the second item follows from the first one and Theorem~\ref{thm:thom_iso_equiv}.
\end{proof}

\begin{remark}
	Suppose that $S$ satisfies \cite[Hypoth\`ese~7.3]{Ayo14}, e.g. $S$ is a variety over a characteristic zero field $k$ of finite cohomological dimension. Then \cite[Remarque~11.3]{Ayo14} yields that the category of \'etale motivic sheaves $\mathrm{D}_{\A^1,\et}(S)$ that can be recovered from $\DAS$ imposing \'etale descent admits normalized Thom isomorphisms for vector bundles and, consequently, for vector bundles with an additional structure.
\end{remark}

\end{document}